\documentclass[12pt,a4paper]{amsart}
\usepackage[T1]{fontenc}
\usepackage{mathpazo}
\usepackage{xcolor}
\usepackage{amsmath}
\usepackage{amsthm}
\usepackage{amsfonts}
\usepackage{amssymb}
\usepackage{graphicx}
\usepackage[all]{xy}
\usepackage[left=2cm,right=2cm,top=2cm,bottom=2cm]{geometry}
\usepackage{tikz-cd}
\usepackage{bm}
\usepackage[pdfauthor={Renaud Coulangeon and Gabriele Nebe},
pdftitle={Slopes of Hermitian lattices, tensor product and group actions}]{hyperref}

\theoremstyle{plain}
\newtheorem*{theo}{Theorem}
\newtheorem{conj}{Conjecture}
\newtheorem*{conj*}{Conjecture}
\newtheorem{theorem}{Theorem}[section]
\newtheorem{proposition}[theorem]{Proposition}

\newtheorem{lemma}[theorem]{Lemma}
\newtheorem{remark}[theorem]{Remark}
\theoremstyle{definition}
\newtheorem{defn}[theorem]{Definition}

\newcounter{biscompt}
\newtheorem{conjbis}[biscompt]{Conjecture}

\newcommand{\C}{\mathbb{C}}
\newcommand{\NN}{\mathbb{N}}
\newcommand{\Aut}{\operatorname{Aut}}
\newcommand{\Id}{\operatorname{Id}}
\newcommand{\rk}{\operatorname{rk}}
\newcommand{\Hom}{\operatorname{Hom}}
\newcommand{\End}{\operatorname{End}}

\newcommand{\GL}{\operatorname{GL}}
\newcommand{\im}{\operatorname{Im}}
\newcommand{\vol}{\operatorname{vol}}

\newcommand{\N}{\operatorname{N}}

\newcommand{\eg}{\emph{e.g. }}
\newcommand{\ie}{\emph{i.e. }}

\theoremstyle{definition}
\newtheorem*{defn*}{Definition}

\begin{document}

\title{Slopes of Euclidean lattices, tensor product and group actions}

\author{Renaud Coulangeon and Gabriele Nebe}
\address{%
	Univ. Bordeaux, CNRS, IMB, UMR 5251,
	33405 Talence cedex, France} 
\email{Renaud.Coulangeon@math.u-bordeaux.fr}

\address{%
Lehrstuhl D f\"ur Mathematik,
RWTH Aachen, 
D-52056 Aachen }
\email{nebe@math.rwth-aachen.de}

\date{\today}
\maketitle
	
	\begin{abstract}
		We study the behaviour of the minimal slope of Euclidean lattices (or more generally ${\mathcal O}_K$-lattices)  under tensor product. A general conjecture predicts that 
		$$\mu_{min}(L \otimes M) = \mu_{min}(L)\mu_{min}(M)$$
		for all lattices $L$ and $M$. We prove that this is the case under the additional assumptions that $L$ and $M$ are acted on multiplicity-free by their automorphism group, such that one of them 
		has at most $2$ irreducible components.
	\end{abstract}
	\section*{Introduction}
	
	The notion of \emph{slope} and the related concept of \emph{(semi)stability} were initially introduced to study vector bundles over smooth projective curves \cite{MR0175899,MR0184252}. Since then, similar concepts have appeared in a wide range of mathematical contexts, often by analogy with the original geometric setting. 
	In these various theories, one can define a \emph{canonical filtration} of any object by semistable ones, a fact first recognized by G. Harder and M. Narasinham in the case of vector bundles on curves \cite{MR0364254}. A \emph{canonical polygon} is associated to this filtration, from which the terminology of "slopes" arise. The similarities between the many occurrences of these \emph{slope filtrations} strongly appealed for a general theoretical framework, which was developed more recently by Y. Andr\'e \cite{MR2571693} and H. Chen \cite{MR2559690}. \medskip

	This formalism applies in particular to Euclidean lattices, as observed by U. Stuhler in \cite{MR0424707,MR0447126}. This program was pushed further by D. Grayson \cite{MR780079,MR870711}, who laid the foundation of the theory in the more general context of $\mathcal{O}_K$-lattices, $\mathcal{O}_K$ being the ring of integers of a number field $K$, and showed that these ideas may be used as an alternative to Borel-Serre compactification to prove the finite presentation of arithmetic groups and the finite generation of their (co)homology. An Arakelov version of this theory, in terms of Hermitian vector bundles, was built by J.-B. Bost in the 1990s, and has been much studied since then. Sticking to the case of ordinary Euclidean lattices, 
	the relevant notions are defined as follows : 
	\begin{itemize}
		\item the slope $\mu(L)$ of a nonzero Euclidean lattice $L$ is the quantity $\left( \vol L\right)^{1/\dim L}$ ; 
		\item its minimal slope $\mu_{min} (L)$ is the minimum of the slopes of all its non zero sublattices ;
		\item a lattice is \emph{semistable} if its slope and minimal slope coincide, or equivalently, if its canonical filtration is trivial. 
	\end{itemize}Note that we adopt here a multiplicative version of the slope, following Stuhler, whereas more recent works, in compliance with the geometric origin of the theory, rather define the slope as $\dfrac{-\log \vol L}{\dim L}=-\log \mu(L)$, and accordingly consider \emph{maximal} instead of \emph{minimal} slope. \medskip 

Quite naturally, one would like to understand the behaviour of slopes and semistability under standard algebraic operations (direct sum, exact sequences, duality, tensor product). In this respect, the tensor product is rather enigmatic. It is known that the tensor product of semistable vector bundles on a smooth algebraic curve in characteristic zero is semistable \cite{MR0184252}. Conservation of semistability under tensor product is also known to hold in several other contexts where a slope filtration is available, but it can apparently not be explained using general formal arguments : in all cases an \emph{ad hoc} proof is needed, often difficult. In \cite{MR2571693}, Y. Andr\'e coined the term "tensor-multiplicative" to qualify slope filtrations having this property. Surprisingly, the question as to whether tensor-multiplicativity holds for Euclidean lattices is still largely open. Recall that the tensor product of two inner product spaces $E$ and $F$ can be turned into an inner product space by setting 
	\begin{equation}\label{tp}
	x\otimes y \cdot x'\otimes y' = (x\cdot x')(y \cdot y') 
	\end{equation}
	 for all $x, x'$ in $E$ and $y, y'$ in $F$ and extending \eqref{tp} to $E\otimes F$ by bilinearity. If $L\subset E$ and $M \subset F$ are two Euclidean lattices, their tensor product $L \otimes M$ thus inherits a structure of Euclidean lattice int $E \otimes F$ equipped with the above inner product. In a seminar held in Oberwolfach in July 1997, J.-B. Bost conjectured that the slope filtration of Euclidean lattices is tensor-multiplicative. The initial conjecture was formulated in the wider context of Hermitian vector bundles, or $\mathcal{O}_K$-lattices in the terminology of Grayson. Extending the notions of slopes to this more general context, the conjecture reduces to the following statement: 
	\begin{conj}\label{bc} 
The minimal slope of the tensor product of two $\mathcal{O}_K$-lattices $L$ and $M$ is equal to the product of their respective minimal slopes.
	\begin{equation}\label{star}
		\mu_{min}(L \otimes_{\mathcal{O}_K} M) = \mu_{min}(L)\mu_{min}(M).
	\end{equation}
	\end{conj} 
An equivalent formulation of this conjecture, in terms of semistability, is as follows :
\begin{conjbis}\label{unbis}
The tensor product of two semistable lattices is semistable.
\end{conjbis}
The inequality $\mu_{min}(L \otimes M) \leq \mu_{min}(L)\mu_{min}(M)$ is clear, so the question is whether this inequality can be strict. The equivalence between Conjecture \ref{bc} and its variation \ref{unbis} is well-known to the experts, although it is somewhat difficult to find a reference for this in the literature (a short argument is presented in Section \ref{group} Proposition \ref{reduction}). 

The hope that the known proofs of the tensor-multiplicativity for vector bundles over curves could inspire a proof of Conjecture 1 is probably overoptimistic, as it was very precisely analyzed by Y. Andr\'e in \cite{MR2872959}. From another viewpoint, this conjecture is reminiscent of the problem of the first minimum of a tensor product : denoting by $\lambda(L)$ the shortest length of a nonzero vector in a lattice $L$, it is clear that 
\begin{equation}\label{fm}
\lambda(L\otimes M) \leq \lambda(L)\lambda(M)
\end{equation}
for all lattices $L$ and $M$ but it is known, by a non constructive argument due to Steinberg, that in large dimensions, there exist lattices for which inequality \eqref{fm} is strict (see \cite[Theorem 9.6]{MR0506372}). On the opposite direction, Kitaoka showed that \eqref{fm} is an equality as long as $L$ or $M$ has dimension $\leq 42$ (see \cite{MR0444567}, \cite[chapter 7]{MR1245266}), which makes the construction of an explicit example of lattices $L$ and $M$ for which \eqref{fm} is not an equality quite difficult, if ever possible.

 This analogy would apparently advocate against Conjecture \ref{bc} : one could expect \eqref{star} to be true in small dimensions, but false in general for high dimensional lattices.
However, our understanding of the minimal slope of a tensor product is at the moment quite modest : the validity of \eqref{star} has been established only in very small dimensions by Bost and Chen \cite{MR3035951}, 
namely when $\dim L \dim M \leq 9$, through a very difficult proof ; on the other hand, almost no general result is known for higher dimensional lattices except some cases where \eqref{star} holds for trivial reasons (\textit{e.g.} when $L$ and $M$ are both unimodular). 
 
The two fundamental obstacles to overcome in order to prove or disprove equality \eqref{star} are the computation of the minimal slope of a lattice, which is difficult in general, and the complexity of the sublattices of a tensor product. These problems can nevertheless be notably simplified in the event that a sufficiently big automorphism group is available. For instance, if the automorphism group of either $L$ or $M$ acts absolutely irreducibly on the underlying space, then \eqref{star} is true (see \cite[Proposition A.3]{MR1423622} or \cite[Proposition 5.1]{MR3096565}). This is based on the crucial observation that the canonical filtration of a lattice is fixed by its automorphisms (see Theorem \ref{basics} \eqref{autinv} below). An intermediate situation is that of a multiplicity-free action. Indeed, this assumption guarantees that the lattices involved possess only finitely many sublattices (up to scalar multiples)
fixed by the corresponding automorphism groups. In this situation, the computation of the minimal slope becomes in principle tractable, whatever the dimension, since it amounts to the inspection of finitely many invariant sublattices. Moreover, if both lattices $L$ and $M$ are acted on multiplicity-free by their automorphism groups, then it is easily seen (see Proposition \ref{mf}) that the minimal slope of $L\otimes_{\mathcal{O}_K} M$ is achieved on a subspace $\mathcal{U}$ which is "split", \ie of the shape
\begin{equation}\label{split}
U =\bigoplus_{i=1}^rE_i\otimes F_i
\end{equation}
where the $E_i$s and $F_i$s are subspaces of $K L$ and $K M$ respectively.
This is of course a very specific situation, but nevertheless, even in that case, the verification of \eqref{star} is nontrivial, except if $r=1$. Thus, it seems natural to study the following special case of Conjecture \ref{bc}:
\begin{conj}\label{wbc}
	Let $L$ and $M$ be lattices acted on multiplicity-free by their automorphism groups. Then $\mu_{min}(L \otimes M) = \mu_{min}(L)\mu_{min}(M)$.
\end{conj} 

Our main result is precisely a proof of the first nontrivial case of Conjecture \ref{wbc}, that is when $r=2$ in \eqref{split}. For technical reasons which will be made clear in Sections \ref{group} and \ref{main}, we have to restrict ourselves to totally real or CM-fields. These fields are the most important ones when dealing with representations of finite groups as all character fields are subfields of cyclotomic fields and hence totally real or CM-fields. Moreover all complex representations are equivalent to representations over such fields. In sum, we will establish the following result:

\begin{theo}
	Let $K$ denote either a totally real or a CM number field, and let $L$ and $M$ be $\mathcal{O}_K$-lattices acted on multiplicity-free by their automorphism groups, such that one of them has 
	at most two irreducible components. Then $\mu_{min}(L \otimes_{\mathcal{O}_K} M) = \mu_{min}(L)\mu_{min}(M)$.
\end{theo}

This result, and more generally Conjecture 2 (if true), partly offsets the impression, based on the analogy with the first minimum, that Conjecture 1 could be false. \medskip

Here is an outline of the paper: in Section \ref{prel} we recall, essentially without proofs, the basics about slopes of lattices. General considerations about tensor products, group actions and slope filtration are developed in Section \ref{group}. The final section is devoted to the proof of our main result.

\section*{Acknowledgement}
The first author thanks Ga\"el R\'emond for his valuable explanations about slope filtrations during the conversations he had with him on this topic. The authors also express their sincere gratitude to the anonymous referee for his valuable comments who helped to improve the presentation of the paper.

This material is partially based upon work supported by the National Science Foundation under Grant No. DMS-1439786 while the first author was in residence at the Institute for Computational and Experimental Research in Mathematics in Providence, RI, during the Spring 2018 semester.
\section*{Notation} 

 To avoid confusions between ordinary and orthogonal direct sums, we will use the symbol $\perp$ for the latter and $\oplus$ for the former. 
\section{Preliminaries on ${\mathcal O}_K$-lattices and slopes}\label{prel}

General references for this introductory section are \cite{MR0424707,MR780079,MR2127941,MR2571693,MR2872959}. 

Let $K$ be a number field and denote by 
${\mathcal O}_K$ its ring of integers. We first recall the basic notation and results for ${\mathcal O}_K$-lattices (see \cite{MR780079}) or equivalently Hermitian vector bundles over Spec$({\mathcal O}_K)$ in the terminology of \cite{MR3035951}. 

Let $E$ be a  vector space over $K$ of dimension, say, $\ell \in \NN $ such that 
for all infinite places $\sigma : K \hookrightarrow \C $ the complex vector space
$E_{\sigma } := E \otimes _K \C $ comes equipped with a Hermitian form $h_{\sigma }$. 
Then a ${\mathcal O}_K$-lattice or equivalently 
a Hermitian vector bundle over Spec$({\mathcal O}_K)$ in $E$ is the data $(L, (h_{\sigma } )_{\sigma } )$ of a finitely generated 
projective  ${\mathcal O}_K$-submodule $L$ of $E$ containing a $K$-basis of $E$ together with the collection of Hermitian forms $h_{\sigma}$. 
The dimension of $E$ is also called the rank  of $L$, $\rk (L) := \ell $. It follows from the well-known classification of modules over Dedekind rings  (see \eg \cite[Th. 81:3]{MR1754311}) that any ${\mathcal O}_K$-lattice admits a \textit{pseudo-basis} $\left(\mathfrak a_i, b_i\right) _{1 \leq i \leq \ell}$, where $\mathfrak a_1, \dots \mathfrak a_{\ell}$ are fractional ideals, and $\left\lbrace b_1, \dots, b_{\ell} \right\rbrace$ is a $K$-basis of $E$ such that
\begin{equation}\label{psb} 
L=\sum_{i=1}^\ell \mathfrak a_i b_i.
\end{equation} 
Following Grayson we adopt the multiplicative notation and define the 
{\em volume} of $L:=(L, (h_{\sigma } )_{\sigma } ) $ as 
$$\vol (L) = \vol ((L, (h_{\sigma } )_{\sigma } ) = \N\left( \prod_{i=1}^{\ell}\mathfrak a _i \right) \prod _{\sigma } 
\det(h_{\sigma }(b_i,b_j))^{e_{\sigma}/2}  $$
where $e_{\sigma}=1$ or $2$ according to whether $\sigma$ is real or complex.
\begin{remark}\label{rk1}
	If $K$ is either a totally real or a CM field, then one can consider ${\mathcal O}_K$-lattices of a particular shape : suppose that the $K$-vector space $E$ is endowed with a totally positive definite Hermitian form $h$, \ie $h$ is a $K$-valued Hermitian form on $E$ such that the extensions ${\sigma}(h)$ of $h$ to all completions $E_{\sigma}$ at infinite places are positive definite. Then, any full-rank  finitely generated 
	projective  ${\mathcal O}_K$-submodule $L$ of $E$ inherits a structure of ${\mathcal O}_K$-lattice $(L,(h_{\sigma })_{\sigma }) $ 
with  $h_{\sigma} :=\sigma(h)$. In what follows, we call ${\mathcal O}_K$-lattices of that type \emph{$K$-rational ${\mathcal O}_K$-lattices}.
\end{remark} 
In the sequel, in accordance with the convention in \cite{MR780079}, the word sublattice stands for "primitive" (or "pure") sublattice (recall that a sub ${\mathcal O}_K$-module $M$ of 
an ${\mathcal O}_K$-module $L$ is primitive if the quotient $L/M$ is torsion free, or equivalently if $M$ is a direct summand of $L$). 
We will have to consider in places non primitive submodules of a given lattice $L$, in particular submodules of finite index, and will consequently refrain from using the term "sublattice" in such situations. 
	
If $M$ is a sublattice of $L$, then the quotient $L/M$ is also an ${\mathcal O}_K$-lattice, 
the inner products being inherited from the identification of $(L/M \otimes K) \otimes _{\sigma } \C $ with $((M\otimes K)\otimes_{\sigma } \C)^{\perp}$. 
By a morphism of lattices, we mean a morphism $f$ of the underlying ${\mathcal O}_K$-modules, with operator norm $ \leq 1$.
	
The notion of exact sequence is defined accordingly. For instance, whenever $M$ is a sublattice of $L$, we have an exact sequence
		\[ 0\longrightarrow M\longrightarrow L \longrightarrow L \slash M\longrightarrow 0 .\]
In full generality, the dual $L^{\vee}=\Hom_{{\mathcal O}_K}(L, {\mathcal O}_K)$ of an ${\mathcal O}_K$-lattice in $E$ can be viewed as an ${\mathcal O}_K$-lattice in $E^{\vee}=\Hom_K(E,K)$, for suitable Hermitian metrics $h^{ \vee} _{\sigma}$
characterized by the property
\begin{equation}
h^{ \vee}_{\sigma}(y^{\vee},y^{\vee})=\sup_{0 \neq x \in E}\dfrac{\vert y^{\vee}(x)\vert_{\sigma}^2}{h_{\sigma}(x,x)} \ \text{ for all } y \in E.
\end{equation}
In the particular situation of Remark \ref{rk1}, that is when $K$ is either a totally real or a CM field and the underlying $K$-vector space $E$ is 
endowed with a totally positive definite $K$-rational
Hermitian form $h$, then one can alternately define the dual  of an ${\mathcal O}_K$-lattice in $E$ as
\begin{equation}
L^{*}= \left\lbrace y \in E \mid h\left(L,y\right) \subset {\mathcal O}_K \right\rbrace.
\end{equation}
With this point of view, $L$ and its dual $L^{*}$ are ${\mathcal O}_K$-lattices with the same underlying vector space $E$ (one readily checks that $L^{*}$ and $L^{\vee}$ are isomorphic as ${\mathcal O}_K$-lattices). 

\begin{defn}
\begin{enumerate}
	\item 	The \emph{slope} 
of a lattice $L$ is defined as 
\[ 	\mu (L)=\left( \vol L\right)^{1/\rk L} 
. \]
\item The \emph{minimal} and \emph{maximal} slopes of a lattice are defined as 
 \[ \mu_{min} (L)=\min_{{0} \neq M \subset L} \mu (M) \]
and
\[ \mu_{max}(L)=\max_{N \subsetneq L} \mu(L\slash N) .\]
\end{enumerate}
\end{defn}
	\bigskip
	
We collect in the following lemma some elementary properties of slopes, for further reference. 
\begin{lemma} \label{elem}Let $L$, $M$, and $N$ be ${\mathcal O}_K$-lattices of rank $\ell = \rk (L)$, $m=\rk(M)$, and $n=\rk (N)$.
\begin{enumerate}
	\item If $M$ is an ${\mathcal O}_K$-submodule of $L$ of finite index, then $\ell = m$ and $\left[L : M\right]=\left( \dfrac{\mu(M)}{\mu(L)} \right)^{\ell}$.
	\item $\mu(L^{\vee })=\mu(L)^{-1}$ .
	\item $\mu (L \otimes M) = \mu (L) \mu (M)$ .
	\item \label{es} For any exact sequence 
		\[ 0\longrightarrow M\longrightarrow L \longrightarrow N \longrightarrow 0 \]
	one has 
		\[ \mu(L)^{\ell }= \mu (M) ^{m}\mu (N)^{n} . \]		
	Consequently, 
 \[ 	\min(\mu(M),\mu(N))	\leq \mu(L) \leq \max(\mu(M),\mu(N)) \]
	and both inequalities are strict, unless $\mu(M)=\mu(N)=\mu(L)$.
	\item In particular
	\begin{align*}
		\mu(L/M)^{\ell-m}&=\mu(L)^{\ell}\mu (M)^{-m}\\
		\mu (M \perp N)^{m+n}&= \mu (M) ^{m}\mu (N)^{ n}. 
	\end{align*}
\end{enumerate}
\end{lemma}
A slightly less immediate property of slopes is the following inequality, which we call the "parallelogram constraint", following Grayson :

 \begin{lemma}\cite[Proposition 2]{MR0424707}\cite[Theorem 12]{MR780079}\label{para}
 	If $L_1$ and $L_2$ are sublattices of a lattice $L$, one has :
 	\[ \mu (L_1\slash L_1 \cap L_2 )\geq \mu (L_1+L_2\slash L_2) .\]
 \end{lemma}

	The invariant $\mu_{min}$ (or $\mu_{max}$) induces a canonical filtration, starting with the so-called \emph{destabilizing sublattice}. We collect below without proofs some properties of this filtration, which we name \emph{Grayson-Stuhler filtration} in reference to its discoverers (the proofs essentially rely on repeated applications of Lemma \ref{para}).
	\begin{theorem}\cite[Satz 1]{MR0424707}	\cite{MR780079}\label{basics}
		Let $L$ be an ${\mathcal O}_K$-lattice in $E$.
	\begin{enumerate} 
		\item The set of sublattices $M\subset L$ such that $\mu(M)=\mu_{min}(L)$, admits a \emph{maximum}, with respect to inclusion, called the \emph{destabilizing sublattice of $L$ }.
		\item The set of sublattices $N \subset L$ such that $\mu(L\slash N)=\mu_{max}(L)$ admits a \emph{minimum}, with respect to inclusion, called the \emph{co-destabilizing sublattice of $L$ }.
		\item To any lattice $L$ is associated a \emph{canonical filtration}
		\[ \left\lbrace 0 \right\rbrace=L_{(0)} \subsetneq L_{(1)} \subsetneq L_{(2)} \subsetneq \dots \subsetneq L_{(m)}=L \] 
		defined recursively by the conditions that
		\begin{enumerate}
			\item $L_{(1)}$ is the destabilizing sublattice of $L$.
			\item $\left\lbrace 0 \right\rbrace=L_{(1)}\slash L_{(1)} \subsetneq L_{(2)}\slash L_{(1)} \subsetneq\dots \subsetneq L_{(m)}\slash L_{(1)} = L\slash L_{(1)}$ is the canonical filtration of $L\slash L_{(1)}$.
		\end{enumerate}
		\item With the previous notation, $L_{(1)}$ and $L_{(m-1)}$ are respectively the destabilizing and co-destabilizing sublattices of $L$.
		\item \label{autinv}The canonical filtration is invariant under any automorphism of $L$.
		\item\label{seis} One has $\mu(L_{(i)}\slash L_{(i-1)}) < \mu(L_{(i+1)}\slash L_{(i)})$ for any $1 \leq i \leq m-1$ and this set of inequalities characterizes the canonical filtration. 
	\end{enumerate}
	\end{theorem}
	
	Note that, since a sublattice $M$ of $L$ is entirely determined by the subspace $F =KM := M\otimes _{{\mathcal O}_K} K $ it generates in $E=KL$, the canonical filtration can be viewed as a filtration by vector \emph{subspaces} of the underlying vector space $E$, namely 
	 \[ \left\lbrace 0 \right\rbrace \subsetneq F_{(1)} \subsetneq F_{(2)} \subsetneq \dots \subsetneq F_{(m)}=E \] 
	where $ F_{(i)}=K L_{(i)}$ and, conversely, $L_{(i)}=L \cap F_{(i)}$.
One can consequently speak of the destabilizing (resp. co-destabilizing) \emph{subspaces} of $E$ with respect to $L$.

Although we won't develop further this point of view, one should also point out the geometric interpretation of this canonical filtration, in terms of the associated \emph{canonical polygon}, which justifies the terminology of slope (resp. minimal and maximal slope) see 	\cite{MR780079}.\medskip
	
The notion of semistability can be formulated using the previous proposition :
\begin{defn}
 A lattice $L$ is semistable if it satisfies one of the following equivalent conditions :
 \begin{enumerate}
 	\item $\mu(M) \geq \mu(L)$ for every sublattice $M$ of $L$ .
 	\item $L$ coincides with its destabilizing sublattice $L_{(1)}$.
 	\item The co-destabilizing sublattice $L_{(m-1)}$ of $L$ is reduced to $\left\lbrace0\right\rbrace$.
 \end{enumerate}
\end{defn}

We now review the properties of $\mu$ and $\mu_{min}$ with respect to quotients and duality.

 If $M$ is any sublattice of $L$ , then to the exact sequence
 \[ 0\longrightarrow M\longrightarrow L \longrightarrow L \slash M\longrightarrow 0 \] corresponds an exact sequence
 \[ 0\longrightarrow \left( L \slash M \right)^{\vee } \longrightarrow L^{\vee } \longrightarrow M^{\vee }\longrightarrow 0 .\] 
 
Consequently, setting $M^{\sharp}:= \im \left( \left( L \slash M \right)^{\vee } \longrightarrow L^{\vee } \right)$ , we have
\begin{equation}
\mu(L^{\vee }\slash M^{\sharp}) = \mu(M^{\vee })=\mu(M)^{-1}. 
\end{equation}

As a consequence, we get the following proposition :
\begin{proposition}\label{dual}
	The map 
	\[ M \mapsto M^{\sharp}:= \im \left( \left( L \slash M \right)^{\vee } \longrightarrow L^{\vee } \right) \]
	induces a bijection between the sets of sublattices of $L$ and $L^{\vee }$ respectively, which exchanges the destabilizing and co-destabilizing sublattices of $L$ and $L^{\vee }$. 
	
	In particular, \[ \mu_{max}(L^{\vee })=(\mu_{min}(L))^{-1}. \] More generally, if 
	\[ \left\lbrace 0 \right\rbrace=L_{(0)} \subsetneq L_{(1)} \subsetneq L_{(2)} \subsetneq \dots \subsetneq L_{(m)}=L \] 
	is the canonical filtration of $L$ then
	\[ \left\lbrace 0 \right\rbrace = L_{(m)}^{\sharp}\subsetneq L_{(m-1)}^{\sharp} \subsetneq L_{(m-2)}^{\sharp} \subsetneq\dots \subsetneq L_{(1)}^{\sharp} \subsetneq L_{(0)}^{\sharp}=L^{\vee } \] 
	is that of $L^{\vee }$. 
	In particular, $L$ is semistable if and only if $L^{\vee }$ is.
\end{proposition}
\begin{remark}\label{rk2}
		In the situation of Remark \ref{rk1}, \ie when $K$ is a totally real or CM number field and $L$ is a $K$-rational ${\mathcal O}_K$-lattice, then one can identify $L \slash M$ with $\pi_{F^{\perp}}(L)$, where $F=KM$ and orthogonality is with respect to the Hermitian inner product on $E=KL$. Under this identification, the lattice $M^{*}$ identifies with $\pi_F(L^{*})$ and $M^{\sharp}$ with $L^{*}\cap F^{\perp}$.
\end{remark}

We end this section with a last useful property of $\mu_{min}$ with respect to quotients, communicated to the first author by Ga\"el R\'emond :
\begin{lemma}\label{x}
	If $M \subset L$ then 	$$\mu_{min}(M) \geq \mu_{min}(L) \geq \min(\mu_{min}(M) ,\mu_{min}(L/M)).$$
	In particular, if $\mu_{min}(M) \leq \mu_{min}(L/M)$ then $\mu_{min}(L)=\mu_{min}(M)$.	
\end{lemma}
\begin{proof}
	Let $L'$ be the destabilizing sublattice of $L$; one has the short exact sequence
	\begin{equation}
	0 \rightarrow L' \cap M \rightarrow L' \rightarrow L'\slash L' \cap M\rightarrow 0
	\end{equation}
	from which we derive the inequality
	\begin{equation}
	\mu_{min}(L)=\mu(L') \geq \min \left(\mu(L'\cap M), \mu (L'\slash L' \cap M )\right).
	\end{equation}
	On the other hand, from Lemma \ref{para}, one has
	\begin{equation}
	\mu (L'\slash L' \cap M )\geq \mu (L'+M\slash M)\geq \mu_{min}(L\slash M).
	\end{equation} 
	Finally, 
	\begin{itemize}
		\item 	if $\mu(L'\cap M)\leq \mu (L'\slash L' \cap M )$, then $\mu_{min}(L)=\mu(L') \geq \mu(L'\cap M) \geq \mu_{min}(M)$, 
		\item if $\mu(L'\cap M)\geq \mu (L'\slash L' \cap M )$, then $\mu_{min}(L)=\mu(L') \geq \mu (L'\slash L' \cap M) \geq \mu (L'+M\slash M)\geq \mu_{min}(L\slash M)$. 
	\end{itemize}
	In all cases, one has 
	$$\mu_{min}(M) \geq \mu_{min}(L) \geq \min(\mu_{min}(M) ,\mu_{min}(L/M)).$$ 	
\end{proof}	

\section{Grayson-Stuhler filtration in the presence of a group action}\label{group}
This section, still preparatory, gives some structural properties of the destabilizing sublattice of lattices endowed with group actions and their tensor products. 

\subsection{Review on the tensor product of group representations}\label{two}

\begin{defn}\label{abstmult} Let $K$ be a field and $G$ a group.
	\begin{enumerate}
		\item A $K[G]$-module $E$ is \emph{absolutely irreducible} if the $\overline{K}[G]$-module $E_{\overline{K}}:=E\otimes_K \overline{K}$
			is irreducible. 
			Here $\overline{K}$ denotes an algebraic closure of $K$.
		\item A $K[G]$-module $E$ is is \emph{multiplicity-free} if it splits as a direct sum $E=\bigoplus_{i=1}^r E_i$ of pairwise non isomorphic \emph{absolutely irreducible} $K[G]$-submodules $E_i$.
	\end{enumerate}
\end{defn}
The following proposition is certainly classical. We nevertheless include it, together with its proof, for sake of completeness.
\begin{proposition}\label{mf}
	Let $G$, $H$ be finite groups, and $K$ a field of characteristic either zero or prime to both $\vert G\vert$ and $\vert H \vert$. Let $E$ be a finite dimensional $K[ G]$-module, $F$ a finite dimensional $K[H]$-module. 
	
	If $E$ is a multiplicity-free $K[G]$-module, \ie $E=\bigoplus_{i=1}^r E_i$ where the $E_i$s are pairwise non isomorphic absolutely irreducible $K[G]$-submodules, then for any $K[G \times H]$- submodule $U$ of $E\otimes F$, there exist 
	$K [H]$-submodules $F_1$, \dots, $F_r$ of $F$ (possibly zero)
	such that
	\begin{equation*}
	U=\bigoplus_{i=1}^rE_i\otimes F_i .
	\end{equation*}
	
\end{proposition}	

Note that the $E_i$s and $F_i$s play asymmetrical roles in the above statement : in particular, the $F_i$s are not irreducible in general, even if $F$ is multiplicity free as a $K[H]$-module, and may intersect each other non trivially.
	
\begin{proof}[Proof of the proposition]
Let us first observe that if $E'$ is an absolutely irreducible $K[G]$-module and $F'$ an irreducible $K[H]$-module, then $E'\otimes F'$ is $G\times H$-irreducible over $K$. Indeed, under these assumptions, one has $\End_{K[G]}(E') = K \Id_{E'}$ (see \eg \cite[Theorem 3.43]{MR632548}) and $\End_{K[H]}(F') = D$ is a division algebra over $K$ (Schur's lemma), so that $\End_{K[G\times H]}(E'\otimes F') = K \Id_{E'}\otimes D \simeq D$, whence we deduce that $E'\otimes F'$ is $G\times H$-irreducible over $K$, since the algebra $K[G\times H]$ is semisimple. 

Let now $F=\oplus_j \overline{F}_j$ be the decomposition of $F$ into $H$-isotypic components (note that the $\overline{F}_j$ are not $H$-irreducible in general).
From the previous observation, we deduce easily that the subspaces $E_i \otimes \overline{F}_j$ are the $G \times H$-isotypic components of $E \otimes F$. The proposition will be finally proved if we can show that any $G \times H$-stable subspace $U$ of $E_i \otimes \overline{F}_j$ is equal to $E_i \otimes F_j$ for a suitable $H$-stable subspace $F_j$ of $\overline{F}_j$. The semisimplicity of $K[G\times H]$ insures that such a subspace $U$ admits a $G\times H$-invariant complement $V$, and the projector $p$ corresponding to the direct sum decomposition $U \oplus V$ is an element of $\End_{K[G\times H]} (E_i \otimes \overline{F}_j) \simeq \End_{K[G]} (E_i) \otimes \End_{K[H]}(\overline{F}_j) = K \otimes \End_{K[H]}(\overline{F}_j) $. In other words, $p=\lambda \otimes f$, with $\lambda \in K$ and $f \in \End_{K[H]}(\overline{F}_j)$, so that $U=p(E_i \otimes \overline{F}_j)=E_i \otimes f(\overline{F}_j)$ is of the required form $E_i\otimes F_j$, setting $F_j=f(\overline{F}_j)$.
\end{proof}

\subsection{Application to the canonical filtration of a tensor product} \label{2dot2}
We will apply the above general considerations to derive some properties of the destabilizing sublattice of a tensor product. The automorphism group of an ${\mathcal O}_K$-lattice 
$(L,(h_{\sigma})_{\sigma }) $ consist on the ${\mathcal O}_K$-module automorphisms of $L$ 
 that additionally preserve all the Hermitian inner products $h_{\sigma}$. 
	For instance, if $K$ is either a totally real or a CM number field and $L$ is a $K$-rational ${\mathcal O}_K$-lattice, in the sense of Remark \ref{rk1}, then its automorphism group consists on the unitary automorphisms of the underlying Hermitian space $KL$ that fix the ${\mathcal O}_K$-module $L$.
	
Let $G$ and $H$ be the automorphism groups of $L$ and $M$.
	 When dealing with Conjecture \ref{bc}, we make no further assumption on $G$ and $H$ (they might even be trivial), while in the case of Conjecture \ref{wbc}, we assume that $G$ and $H$ act multiplicity-free on the underlying vector spaces $E=K L$ and $F=K M$ respectively. Note also that, in this context, since $G$ and $H$ consist of unitary automorphisms, the isotypic components of $E$ and $F$ are mutually orthogonal with respect to the inner products $h_{\sigma}$.

The next proposition, of interest in itself, shows in particular that it is enough to consider Conjecture \ref{bc} and \ref{wbc} for semistable lattices (in other words : Conjecture \ref{bc} and \ref{unbis} are equivalent, as already mentioned in the introduction).

\begin{proposition}\label{reduction}
	Let $(L,M)$ be a \emph{minimal} counterexample to either Conjecture \ref{bc} or Conjecture \ref{wbc} (by "minimal", we mean : "such that $\dim L + \dim M$ is minimal"). 
	Denote by $E$ and $F$ respectively the vector spaces spanned by $L$ and $M$ respectively, and by $U$ the destabilizing subspace of $E\otimes F$ with respect to $L\otimes M$. Then :
	\begin{enumerate}
		\item \label{un}$L$ and $M$ are semistable.
		\item \label{deux} Assume moreover that $K$ is either a totally real or a CM number field, and  that $L$ and $M$ are $K$-rational ${\mathcal O}_K$-lattices. If $U$ splits as $U=\bigoplus_{i=1}^rE_i\otimes F_i$ where the $E_i$s are pairwise orthogonal subspaces of $E$, and the $F_i$s are subspaces of $F$, then $$\sum_{i=1}^{r} E_i =E, \quad \sum_{i=1}^{r}F_i =F \text{ and } \bigcap_{i=1}^{r} F_i =\left\lbrace 0 \right\rbrace.$$
	\end{enumerate}
\end{proposition}
\noindent \emph{Remark.} The situation in Proposition \ref{reduction} \eqref{deux} 
is of course very specific : there is a priori no reason that the destabilizing subspace of $E\otimes F$ with respect to $L\otimes M$ admit a splitting of the form $U=\bigoplus_{i=1}^r E_i\otimes F_i$ in general. 

\begin{proof}
\begin{enumerate}
	\item The reduction to the semistable case is classical and was explained to the first author by Ga\"el R\'emond. We reproduce the argument here, by lack of an appropriate reference. Assume, by way of contradiction, that the pair $(L,M)$ violate either Conjecture \ref{bc} or \ref{wbc}, and that $L$ is not semistable. Let $L' \subsetneq L$ be the destabilizing sublattice of $L$. Thanks to Lemma \ref{x} , we have:
\begin{equation}\label{cex}
	\mu_{min}(L'\otimes M)\geq \mu_{min}(L\otimes M) \geq \min\left( \mu_{min}(L'\otimes M),\mu_{min}(L \slash L' \otimes M ) \right).
\end{equation}
We deduce from the minimality assumption on the pair $(L,M)$ that both $(L',M)$ and $(L\slash L',M)$ satisfy Conjecture \ref{bc} or \ref{wbc} respectively (note that for the latter, we use the fact that the multiplicity free assumption is preserved for $G$-submodules and quotients). Consequently, 
\[ \mu_{min}(L'\otimes M)=\mu_{min}(L')\mu_{min}(M)=\mu_{min}(L)\mu_{min}(M) \]
and 
\[ \mu_{min}(L \slash L' \otimes M)=\mu_{min}(L \slash L' )\mu_{min}(M). \]
We also know from Theorem \ref{basics} \eqref{seis} that $\mu_{min}(L \slash L' ) > \mu_{min}(L)$ if $L'$ is the destabilizing sublattice of $L$ ; together with \eqref{cex}, this implies that 
\[ \mu_{min}(L\otimes M) =\mu_{min}(L)\mu_{min}(M) , \]
contradicting our initial assumption.
\item Let $\hat{E}=\sum_{i=1}^{r} E_i$, $\hat{F}=\sum_{i=1}^{r}F_i$, $\hat{L}=L\cap\hat{E}$, and $\hat{M}=M\cap\hat{F}$. The sublattice $\hat{L}\otimes \hat{M}$ is primitive, \ie $\hat{L}\otimes \hat{M} =(L\otimes M) \cap \hat{E} \otimes \hat{F} $, and since $U \subset \hat{E} \otimes \hat{F}$, one has 
\[ \left(L \otimes M\right) \cap U=\left(\hat{L} \otimes\hat{M}\right) \cap U .\]
By the very definition of $U$ we get
\begin{equation}\label{mu1}
\mu_{min}(L \otimes M)=\mu(L \otimes M \cap U)=\mu(\hat{L} \otimes \hat{M} \cap U).
\end{equation}

Assuming that $(L,M)$ violates Conjecture \ref{bc}, we have 
\begin{equation*} 
\mu_{min}(L)\mu_{min}(M) >\mu_{min}(L \otimes M)
\end{equation*}
so that \eqref{mu1}, together with the trivial inequality
\begin{equation*} 
\mu(\hat{L} \otimes \hat{M} \cap U) \geq \mu_{min}(\hat{L} \otimes \hat{M})
\end{equation*}
implies that
\begin{equation}\label{mu3}
\mu_{min}(L)\mu_{min}(M) >\mu_{min}(\hat{L} \otimes \hat{M}).
\end{equation}
Yet, if we assume that either $\hat{E}\neq E$ or $\hat{F}\neq F$, then by the minimality assumption for the pair $(L,M)$, we infer that $(\hat{L}, \hat{M})$ satisfies Conjecture \ref{bc}, so that
\begin{equation}\label{mu4}
\mu_{min}(\hat{L} \otimes \hat{M})=\mu_{min}(\hat{L})\mu_{min}(\hat{M}) \geq \mu_{min}(L)\mu_{min}(M)
\end{equation}
contradicting \eqref{mu3}.

It remains to prove that $\bigcap_{i=1}^{r} F_i =\left\lbrace 0 \right\rbrace$, or alternatively that $\sum_{i=1}^{r}F_i^{\perp} = F$ (note that the corresponding statement for the $E_i$s is automatically satisfied from the assumption that these subspaces are mutually orthogonal). Having proven that $E=\bigoplus_{i=1}^{r} E_i$, and using the pairwise orthogonality of the $E_i$s, a simple calculation yields
\[ U^{\perp}=\left( \bigoplus_{i=1}^rE_i\otimes F_i\right)^{\perp}=\bigoplus_{i=1}^rE_i\otimes F_i^{\perp}. \]
From part (1)  we can moreover assume that $L$ and $M$ are semistable, as well as $L^{*}$ and $M^{*}$. In particular 
\[ \mu_{min}(L^{*})\mu_{min}(M^{*})=\mu(L^{*})\mu(M^{*})=\mu(L^{*}\otimes M^{*}) \]
As $U$ is the destabilizing subspace of $E \otimes F$ with respect to $L\otimes M$, we infer that $U^{\perp}$ is the co-destabilizing subspace of 
$E \otimes F$ with respect to $L^{*}\otimes M^{*}$, so that the destabilizing subspace $V$ of 
$E \otimes F$ with respect to $L^{*}\otimes M^{*}$ is contained in $U^{\perp}=\bigoplus_{i=1}^r E_i\otimes F_i^{\perp}$. In particular, since $U^{\perp}\subsetneq E\otimes F$, as $U \neq \left\lbrace 0 \right\rbrace$, we have 
\begin{equation}\label{key1}
\mu(L^{*}\otimes M^{*}) >\mu_{min}(L^{*} \otimes M^{*})= \mu(L^{*} \otimes M^{*} \cap V)
\end{equation}
by definition of $V$.
Setting $\tilde{F}=\sum_{i=1}^{r}F_i^{\perp}$ and
 $\widetilde{M^{*}}=M^{*}\cap \tilde{F}$, and noticing that $V \subset E \otimes \tilde{F}$, we can reproduce essentially the same argument as before : one has 
 \[ \left(L^{*} \otimes M^{*}\right) \cap V = \left(L^{*} \otimes\widetilde{M^{*}}\right) \cap V \]
 so that
 \begin{equation}\label{key2}
 \mu_{min}(L^{*} \otimes M^{*})= \mu(L^{*} \otimes M^{*} \cap V)=\mu(L^{*} \otimes \widetilde{M^{*}} \cap V) 
 \end{equation}
 and
 \begin{equation}\label{ten}
 \mu(L^{*}\otimes M^{*}) >\mu_{min}(L^{*} \otimes M^{*})=\mu(L^{*} \otimes \widetilde{M^{*}} \cap V) \geq \mu_{min}(L^{*} \otimes \widetilde{M^{*}}) .
 \end{equation}
If $\tilde{F} \neq F$, then $\dim L^{*} + \dim\widetilde{M^{*}} < \dim L + \dim M$ and the minimality assumption on the pair $(L,M)$ again implies that
\begin{equation}\label{onze}
\mu_{min}(L^{*} \otimes \widetilde{M^{*}}) =\mu_{min}(L^{*})\mu_{min}(\widetilde{M^{*}}) \geq \mu_{min}(L^{*})\mu_{min}(M^{*})= \mu(L^{*}\otimes M^{*}),
\end{equation}
contradicting \eqref{ten}. Therefore $\tilde{F}=F$ , or in other words $\bigcap_{i=1}^{r} F_i =\left\lbrace 0 \right\rbrace$.
\end{enumerate} 
\end{proof}

\section{Main result}\label{main}
\begin{theorem}\label{mt}
	Let $K$ denote either a totally real or a CM number field, and let $L$ and $M$ be $\mathcal{O}_K$-lattices, respectively in $E=KL$ and $F=K M$. 
	Let $G\leq \Aut L$ and $H\leq \Aut M$ and assume $E$ and $F$ to be respectively $G$ and $H$ multiplicity free. Denote by $r$ and $s$ the number of irreducible components of $E$ and $F$ respectively. Then, if $\min (r,s)\leq 2$, one has
	\begin{equation*}
		\mu_{\min}(L\otimes M)=\mu_{min}(L)\mu_{min}(M). 
	\end{equation*}
\end{theorem}

For the proof of the theorem it is convenient to additionally assume that 
$L$ and $M$ are $K$-rational ${\mathcal O}_K$-lattices. 
The general case then follows by the strong approximation property as explained in
the following lemma.

\begin{lemma} \label{approx}
Let $L:=(L,(h_{\sigma})_{\sigma })$ be some ${\mathcal O}_K$-lattice and 
 let $G\leq \Aut(L, (h_{\sigma})_{\sigma }) $. 
Then there is a $K$-rational $G$-invariant form $h:L \times L \to K$,
such that $\sigma (h) $ approximates $h_{\sigma }$ simultaneously for 
all $\sigma $.
\end{lemma}

\begin{proof}
Fix a basis $B$ of $KL $ and work with matrices with respect for $B$. 
Then $G\leq \GL_{\ell }(K)$ is a finite matrix group. 
Put 
$${\mathcal F}(G) := \{ H \in K^{\ell \times \ell} \mid  H = \overline{H}^{tr} 
\mbox{ and }
gH\overline{g}^{tr} = H \mbox{ for all } g\in G \} $$ 
where $\overline{\phantom{s}} $ is the complex conjugation 
in the case that $K$ is a CM field 
 (applied to the entries of the matrices) 
and the identity if $K$ is
totally real. Then ${\mathcal F}(G)$ is a finite 
dimensional vector space over the fixed field of $\overline{\phantom{s}}$
As 
$G\leq \Aut(L, h_{\sigma}) $
 the 
Gram matrices  $H_{\sigma }$ of the $h_{\sigma }$ with respect to 
$B$  are Hermitian matrices satisfying 
$$\ (\star ) \ \sigma (g) H_{\sigma } \overline{\sigma(g)}^{tr} = H_{\sigma } 
\mbox{ for all } g\in G,$$ where $\sigma $ and the complex conjugation 
$\overline{\phantom{s}} $ are  applied  entrywise. 
The matrices $H_{\sigma }$ satisfying $(\star )$ form a 
finite dimensional real vector space spanned by the matrices 
$\sigma (H) $, with 
$H\in {\mathcal F}(G) $.
By the strong approximation property for finite dimensional vector spaces 
 we can find $H\in {\mathcal F}(G) $ such 
that $\sigma(H)$ approximates $H_{\sigma }$ simultaneously for all $\sigma $. 
\end{proof}

Now assume that we have proven Theorem \ref{mt} for $K$-rational ${\mathcal O}_K$-lattices.
Assume that there is a minimal counterexample to the general case.
Then there are semistable ${\mathcal O}_K$-lattices $L$ and $M$ as in the Theorem and
a sublattice $N$ of $L\otimes M$ 
with $\mu (N) < \mu (L) \mu (M) $. 
Approximating the complex valued inner products on $L$ and  $M$ by $K$-rational 
inner products  that are invariant under $G$ respectively $H$ close enough, 
the same lattice $N$ will also satisfy 
$\mu (N) < \mu (L) \mu (M) $ with respect to these $K$-rational forms, contradicting 
our proof. 
So in the following we can and will always assume that $L$ and $M$ are $K$-rational ${\mathcal O}_K$-lattices.

The proof of the theorem relies on the reduction allowed by Proposition \ref{reduction}. The following classical lemma will also be used several times :
\begin{lemma}\label{proj} Let $L$ be a $K$-rational ${\mathcal O}_K$-lattice with underlying $K$-vector space $E$.
	
\begin{enumerate}
	\item\label{cross} If $F$ is a subspace of $E$ and $F^{\perp}$ its orthogonal, then one has the following isomorphisms :
	\[ L\slash \left( L\cap F\perp L \cap F^{\perp} \right) \simeq \pi^{}_F(L)\slash L\cap F \simeq \pi_{F^{\perp}}(L)\slash L \cap F^{\perp} \simeq \pi^{}_F(L) \perp \pi_{F^{\perp}}(L) \slash L.
	\]
	\item Let $F_1, F_2,F_3$ be subspaces of $E$, such that $F_3=F_1 \perp F_2$, and let $\pi_1$, $\pi_2$ and $\pi_3$ denote the orthogonal projections onto $F_1$, $F_2$ and $F_3$ respectively. For $i=1,2,3$, set $L_i=L\cap F_i$. 
	\begin{enumerate}
		\item\label{aaa} There are natural injective morphisms
	\[ 	L_3\slash (L_1 \perp L_2) \hookrightarrow \pi_1 L \slash L_1 \text{ and } L_3\slash (L_1 \perp L_2) \hookrightarrow \pi_2 L \slash L_2 . \]
		\item\label{bbb} There are natural surjective  morphisms
			\[ 	\pi_1 L \slash L_1 \twoheadrightarrow (\pi_1 L \perp \pi_2 L)\slash \pi_3L \text{ and } \pi_2 L \slash L_2 \twoheadrightarrow (\pi_1 L \perp \pi_2 L)\slash \pi_3L .\]

	\end{enumerate}
	
\end{enumerate}
\end{lemma}
\begin{proof}
\begin{enumerate}
	\item The kernel of the surjective morphism $L \twoheadrightarrow \pi^{}_F(L)\slash L\cap F$ is clearly equal to $L\cap F\perp L \cap F^{\perp}$ and similarly exchanging $F$ and $F^{\perp}$ , whence the first two isomorphisms. As for the last one, it amounts to show that the injective morphism $\pi^{}_F(L)\slash L\cap F \hookrightarrow \pi^{}_F(L) \perp \pi_{F^{\perp}}(L) \slash L $ is onto, which is clear from the observation that $\pi^{}_F(x) +\pi_{F^{\perp}}(y) \equiv \pi^{}_F(x-y) \mod L$ .
	\item This follows from \eqref{cross}, since $L_3\slash (L_1 \perp L_2) \simeq \pi_i L_3/L_i \hookrightarrow \pi_i L \slash L_i$ for $i=1,2$.
		\item Similarly, by \eqref{cross}, we know that $(\pi_1 L \perp \pi_2 L)\slash \pi_3L$ is isomorphic to $\pi_i L\slash \pi_3L \cap F_i$, for $i=1,2$, whence the conclusion since $\pi_3L \cap F_i \supset L_i$.
\end{enumerate}
\end{proof}
 
\begin{proof}[Proof of Theorem \ref{mt}]
Assume, by a way of contradiction, that there exists a pair of lattices $(L,M)$ satisfying the assumptions of the theorem and violating its conclusion, \ie such that 
\begin{equation}\label{aa}
\mu_{min}(L\otimes M) <\mu_{\min}(L)\mu_{\min}(M).
\end{equation}
By Proposition \ref{reduction} \eqref{un}, we can assume that both $L$ and $M$ are semistable. If $\min(r,s)=1$, that is if either $E$ or $F$ is absolutely irreducible, then the relation $\mu_{\min}(L\otimes M)=\mu_{min}(L)\mu_{min}(M)$ is well-known to hold (see introduction), in contradiction with \eqref{aa}. Thus $\min(r,s)=2$ and we may assume, without loss of generality, that $r=2$. In other words, $E$ splits as 
\begin{equation}\label{d1}
E= E_1 \perp E_2,
\end{equation}
where $E_1$ and $E_2$ are non isomorphic, and consequently mutually orthogonal, $G$-absolutely irreducible subspaces of $E$. Let $U$ be the destabilizing subspace of $E\otimes F$ with respect to $L\otimes M$. As $U$ is $G\times H$-invariant, we know by Proposition \ref{mf} that there exists $H$-stable subspaces $F_1$ and $F_2$, such that
\begin{equation}\label{d2}
	 U=E_1\otimes F_1 \perp E_2\otimes F_2.
\end{equation}

Furthermore, by Proposition \ref{reduction} \eqref{deux} we can assume that $F=F_1 \oplus F_2$. Since each $F_i$ is, by the multiplicity free assumption, a sum of irreducible, pairwise non isomorphic, hence mutually orthogonal, subrepresentations, we infer that $F_1$ and $F_2$ are themselves mutually orthogonal, so that
\begin{equation}\label{d3}
F=F_1 \perp F_2.
\end{equation}

Let $\ell = \dim L$ and $m=\dim M$. We set $L_i=L \cap E_i$ and $M_i=M\cap F_i$ ($i=1,2$) , denote by $\ell_i$ and $m_i$ their respective dimensions and by $\pi_i$ and $\pi'_i$ the orthogonal projection onto $E_i$ and $F_i$ respectively. Then $\ell=\ell_1+\ell_2$, $m=m_1 +m_2$, $L$ contains $L_1 \perp L_2$ and $M$ contains $M_1 \perp M_2$ , both with finite indices
\begin{equation}\label{indices}
a=\left[L: L_1 \perp L_2\right] \ , \ b=\left[M: M_1 \perp M_2\right].
\end{equation} 
The destabilizing sublattice $$P:=U\cap (L\otimes M)$$ of $L\otimes M$ contains $L_1\otimes M_1 \perp L_2\otimes M_2$ with an index $x \geq 1$ .

As a consequence of \eqref{d1}, \eqref{d2} and \eqref{d3}, one has 
\begin{equation}\label{d4}
U^{\perp}=E_1\otimes F_2 \perp E_2\otimes F_1
\end{equation} 
and we infer from Proposition \ref{dual} and Remark \ref{rk2} that the co-destabilizing sublattice of $(L \otimes M)^{*}$ is 
$$P^{\sharp}=(L\otimes M)^{*} \cap U^{\perp} =\left( \pi_{U^{\perp}}(L\otimes M)\right)^{*}.$$

The situation is summarized in the following diagram :\medskip

\begin{equation}\label{diag}
\begin{tikzcd} [column sep=-10]
\ &	\pi^{}_1L\otimes \pi'_1M\perp\ar[d,"x\,' " ,dash] \pi^{}_2L\otimes\pi'_2M&&\pi^{}_1L\otimes \pi'_2M\perp\ar[d,"y\,'" ',dash] \pi^{}_2L\otimes\pi'_1M & \ \\
&	\pi^{}_U(L\otimes M)\ar[dd,"t",dash,dashed] &\perp\ar[d,dash] & \pi^{}_{U^{\perp}}(L\otimes M)\ar[dd,"t"',dash,dashed] & \\
&	\ &L\otimes M\ar[d,dash]& \ & \\
&	P=L\otimes M \cap U\ar[d,"x",dash]&\perp & L\otimes M \cap U^{\perp}\ar[d,"y"',dash] & \\
\arrow[uuuu,bend left,"a^mb^{\ell}",dash,dashed]&	L_1\otimes M_1\perp L_2 \otimes M_2&&L_1\otimes M_2\perp L_2 \otimes M_1 & \arrow[uuuu,bend right,"a^mb^{\ell}" ',dash,dashed]\\
\end{tikzcd} 
\end{equation}

On the left-hand side of the above diagram, we observe that for 
$i=1,2$ the map $\pi^{}_i \otimes \pi'_i$ induces a monomorphism 
$$\dfrac{P}{L_1\otimes M_1\perp L_2 \otimes M_2} \hookrightarrow \dfrac{\pi^{}_iL\otimes \pi'_iM }{ L_i\otimes M_i} $$ (this is Lemma \ref{proj} \eqref{aaa}), 
whence the upper bound 
\begin{equation}\label{trois}
x \leq \min(a^{m_1}b^{\ell_1},a^{m_2}b^{\ell_2}). 
\end{equation}
Set
\begin{equation*}
\alpha_i=	\left( \dfrac{\mu(L_i)}{\mu(L)} \right)^{\ell_i} \text{ and }\ \beta_i=	\left( \dfrac{\mu(M_i)}{\mu(M)} \right)^{m_i}, \ i=1,2. 
\end{equation*}
These quantities satisfy the relations 
\begin{equation}\label{alphabeta}
\alpha_1\alpha_2=a \text{ and } \beta_1\beta_2 = b.
\end{equation}
Indeed, using Lemma \ref{elem} we have
\[ a=\left[L:L_1\perp L_2\right]=\dfrac{\vol (L_1) \vol(L_2)}{\vol (L)} 
= \dfrac{\mu (L_1)^{\ell_1} \mu (L_2) ^{\ell _2}}{\mu(L) ^{\ell_1+\ell_2}} =
\alpha _1 \alpha _2\]

which proves the first identity in \eqref{alphabeta}, the second being identical.

Using Lemma \ref{elem} and the fact that $L$ and $M$ are semistable, the assumption that $\mu(P)< \mu_{min}(L)\mu_{min}(M)$ amounts to say that
\begin{align*}
	x &=\left[P : 	(L_1\otimes M_1\perp L_2 \otimes M_2)\right]\\ 
&=\dfrac{\left( \mu(L_1)\mu(M_1) \right)^{\ell_1m_1}\left( \mu(L_2)\mu(M_2) \right)^{\ell_2m_2}}{\mu(P)^{\ell_1m_1+\ell_2m_2}}\\
&>\dfrac{\left( \mu(L_1)\mu(M_1) \right)^{\ell_1m_1}\left( \mu(L_2)\mu(M_2) \right)^{\ell_2m_2}}{\left( \mu(L)\mu(M) \right)^{\ell_1m_1+\ell_2m_2}}=\alpha_1^{m_1}\beta_1^{\ell_1}\alpha_2^{m_2}\beta_2^{\ell_2}.
\end{align*}

Together with \eqref{alphabeta} this yields 
\begin{equation}\label{quatre}
x > \alpha_1^{m_1}\beta_1^{\ell_1}\alpha_2^{m_2}\beta_2^{\ell_2} = \alpha_1^{m_1-m_2}\beta_1^{\ell_1-\ell_2}a^{m_2}b^{\ell_2}.
\end{equation}

The combination of \eqref{trois} and \eqref{quatre} implies that 
\[ \alpha_1^{m_1-m_2}\beta_1^{\ell_1-\ell_2}a^{m_2}b^{\ell_2} < \min (a^{m_1}b^{\ell_1},a^{m_2}b^{\ell_2}) \]
or equivalently
\begin{equation} \label{cinq}
\alpha_1^{m_1-m_2}\beta_1^{\ell_1-\ell_2} < \min (1,a^{m_1-m_2}b^{\ell_1-\ell_2}).
\end{equation}

As a consequence of the semistability of $L$ and $M$, one has $\alpha_i \geq 1$ and $\beta_i \geq 1$ , $i=1,2$. Together with the relations $\alpha_1\alpha_2=a$ and $\beta_1\beta_2 = b$, it implies that $1 \leq \alpha_1 \leq a$ and $1 \leq \beta_1 \leq b$. Consequently, inequality \eqref{cinq} cannot hold unless $m_1-m_2$ and $\ell_1-\ell_2$ have opposite signs, that is
\begin{equation}\label{sign1}
(m_1-m_2)(\ell_1-\ell_2)<0.
\end{equation}\medskip
 
 We now consider the right-hand side of the diagram, using duality. 
 First, we have an upper bound for $y'=\left[(\pi^{}_1L\otimes \pi'_2M\perp \pi^{}_2L\otimes\pi'_1M ): \pi^{}_{U^{\perp}}(L\otimes M)\right]$, similar to \eqref{trois}, namely 
 \begin{equation}\label{six}
 y' \leq \min(a^{m_1}b^{\ell_2},a^{m_2}b^{\ell_1}),
 \end{equation}
obtained either by using the natural epimorphisms $\pi^{}_1L\otimes \pi'_2M \slash L_1\otimes M_2 \twoheadrightarrow \pi^{}_1L\otimes \pi'_2M\perp \pi^{}_2L\otimes\pi'_1M \slash \pi^{}_{U^{\perp}}(L\otimes M)$ and $\pi^{}_2L\otimes \pi'_1M \slash L_2\otimes M_1 \twoheadrightarrow \pi^{}_1L\otimes \pi'_2M\perp \pi^{}_2L\otimes\pi'_1M \slash \pi^{}_{U^{\perp}}(L\otimes M)$ (see Lemma \ref{proj} \eqref{bbb}), or by dualizing \eqref{trois}.

Furthermore, one has
 \begin{align*}
 y'&=\left[P^{\sharp} : \left( \pi^{}_1L\otimes \pi'_2M\perp \pi^{}_2L\otimes\pi'_1M \right)^{*}\right]\\
 &=\left[\dfrac{\mu\left( \pi^{}_1L\otimes \pi'_2M\perp \pi^{}_2L\otimes\pi'_1M \right)^{*}} {\mu (P^{\sharp})}\right]^{\ell_1m_2+\ell_2m_1}\\
 &=\dfrac{\left( \mu(\pi^{}_1L)\mu(\pi'_2M) \right)^{-\ell_1m_2}\left( \mu( \pi^{}_2L)\mu(\pi'_1M ) \right)^{-\ell_2m_1}}{\mu (P^{\sharp})^{\ell_1m_2+\ell_2m_1}}.
 \end{align*}
 Using Lemma \ref{elem} again, one has
 \begin{align*}
 \mu (P^{\sharp})^{\ell_1m_2+\ell_2m_1}&=\mu(L^{*}\otimes M^{*})^{\ell m}\mu\left( \dfrac{L^{*}\otimes M^{*}}{P^{\sharp}} \right)^{-\left( \ell m -(\ell_1m_2+\ell_2m_1) \right)}\\
 &=\left( \mu(L)\mu(M)\right)^{-\ell m}\mu\left( \dfrac{L^{*}\otimes M^{*}}{P^{\sharp}} \right) ^{-\left( \ell_1m_1+\ell_2m_2 \right)}.
 \end{align*}
 Since $P^{\sharp}$ is the co-destabilizing sublattice of $(L\otimes M)^{*}$,and $L$ and $M$ are semistable, we have 
 \begin{equation*}
 \mu\left( \dfrac{L^{*}\otimes M^{*}}{P^{\sharp}} \right) > \mu_{max}(L^{*})\mu_{max}(M^{*})=\mu(L)^{-1}\mu(M)^{-1}.
 \end{equation*}

 Plugging this into the above calculation of $y'$, we get
 \begin{equation*}
 y' >\dfrac{\left( \mu(L)\mu(M) \right)^{\ell_1m_2+\ell_2m_1}}{\left( \mu(\pi^{}_1L)\mu(\pi'_2M) \right)^{\ell_1m_2}\left( \mu( \pi^{}_2L)\mu(\pi'_1M ) \right)^{\ell_2m_1}}
 \end{equation*}
 whence finally, using \eqref{alphabeta},
 \begin{equation}\label{sept}
 y'>a^{m_1+m_2}b^{\ell_1+\ell_2}\alpha_1^{-m_2}\beta_2^{-\ell_1}\alpha_2^{-m_1}\beta_1^{-\ell_2}=a^{m_2}b^{ \ell_1}\alpha_1^{m_1-m_2}\beta_2^{\ell_2-\ell_1}
 \end{equation}
 where we also used the relations $\mu(\pi^{}_i L)=a^{-\frac{1}{\ell_i}}\mu(L_i)$ and $\mu(\pi'_i M)=b^{-\frac{1}{m_i}}\mu(M_i)$. The combination of \eqref{six} and \eqref{sept} yields
\begin{equation*}
a^{m_2}b^{ \ell_1}\alpha_1^{m_1-m_2}\beta_2^{\ell_2-\ell_1}< \min(a^{m_1}b^{\ell_2},a^{m_2}b^{\ell_1})
\end{equation*} or equivalently
\begin{equation} \label{huit}
\alpha_1^{m_1-m_2}\beta_2^{\ell_2-\ell_1}<\min(1,a^{m_1-m_2}b^{\ell_2-\ell_1}).
\end{equation}

The same argument we used to derive \eqref{sign1} now yields
\begin{equation}\label{sign2}
(m_1-m_2)(\ell_2-\ell_1)<0
\end{equation}
which is incompatible with \eqref{sign1}.

 \end{proof}

\providecommand{\bysame}{\leavevmode\hbox to3em{\hrulefill}\thinspace}

 \end{document}